\documentclass{amsart}
\usepackage{amsfonts,amsmath,amssymb,amsthm,cmtiup} 
\usepackage{mathrsfs}
\usepackage{mathtools}
\usepackage{enumerate}

\newtheorem{theorem}{Theorem}
\newtheorem{lemma}[theorem]{Lemma} 
\newtheorem{corollary}[theorem]{Corollary} 
\newtheorem{proposition}[theorem]{Proposition}

\newcommand{\zero}{{\mathbf 0}}
\newcommand{\sd}{\mathop{\scriptstyle\triangle}}
\newcommand{\IP}{\operatorname{IP}}
 
\theoremstyle{definition} 

\newtheorem*{remark*}{Remark}
\newtheorem{definition}{Definition}
\newtheorem{example}{Example}
\newtheorem*{example*}{Example}

\begin{document}

\author{Ol'ga V. Sipacheva} 
\title{Large Sets in Boolean and Non-Boolean Groups \endgraf and Topology}
\thanks{This work was financially supported by the Russian Foundation for Basic Research (Project
No.~15-01-05369).}
\address{Department of General Topology and Geometry, Lomonosov Moscow State University, Leninskie
Gory 1, Moscow 119991, Russia}
\email{o-sipa@yandex.ru}

\maketitle

Various notions of large sets in groups and semigroups naturally arise in dynamics and 
combinatorial number theory. Most familiar are those of syndetic, thick (or 
replete), and piecewise syndetic sets. Apparently, the term ``syndetic'' was introduced 
by Gottschalk and Hedlund in their 1955 book~\cite{GH-syndetic} in the context of topological 
groups, although syndetic sets of integers have been studied long before (they appear, e.g.,  
in Khintchine's 1934 ergodic theorem). During the past decades, large sets in $\mathbb 
Z$ and in abstract semigroups have been extensively studied. It has turned out that, e.g., 
piecewise syndetic sets in $\mathbb N$ have many attractive properties: they are partition regular 
(i.e., given any partition of $\mathbb N$ into finitely many subsets, at least one of the subsets 
is piecewise syndetic), contain arbitrarily long arithmetic progressions, and are 
characterized in terms of ultrafilters on $\mathbb N$ (namely, a set is piecewise syndetic 
it and only if it belongs to an ultrafilter contained in the minimal two-sided ideal 
of $\beta\mathbb N$). Large sets of other kinds are no less interesting, and they have 
numerous applications to dynamics, Ramsey theory, the ultrafilter semigroup on $\mathbb 
N$, the Bohr compactification, and so on. 

Quite recently Reznichenko and the author have found yet another application of large sets. Namely, 
we introduced special large sets in groups, which we called fat, and applied 
them to construct a discrete set with precisely one limit point in any countable nondiscrete  
topological group in which the identity element has nonrapid filter of neighborhoods. Using this 
technique and special features of Boolean groups, we proved, in particular, 
the nonexistence of a countable nondiscrete extremally 
disconnected group in~ZFC (see~\cite{arXiv}). 

In this paper, we study right and left thick, syndetic, piecewise
syndetic, and fat sets in groups (although they can be defined for arbitrary 
semigroups). Our main concern is the 
interplay between such sets in Boolean groups. We also consider natural topologies 
closely related to fat sets, which leads to interesting relations between fat sets and 
ultrafilters. 

\section{Basic Definitions and Notation}

We use the standard notation $\mathbb Z$ for the group of integers, $\mathbb N$ for the set (or 
semigroup, depending on the context) of positive integers, and $\omega$ for the set of nonnegative 
integers or the first infinite cardinal; we identify cardinals with the corresponding 
initial ordinals. Given a set $X$, by $|X|$ we denote its cardinality, by $[X]^k$ for 
$k\in \mathbb N$, the $k$th symmetric power of $X$ (i.e., the set of all $k$-element subsets of 
$X$), and by $[X]^{<\omega}$, the set of all finite subsets of $X$.

\begin{definition}[see \cite{BHM}] 
Let $G$ be a group. A set $A \subset  G$ is said to be
\begin{enumerate}
\item[(a)] 
\emph{right thick}, or simply \emph{thick} if, for every finite $F \subset S$, there exists a 
$g \in G$ (or, equivalently, $g\in A$~\cite[Lemma~2.2]{BHM}) such that $Fg \subset A$; 
\item[(b)] 
\emph{right syndetic}, or simply \emph{syndetic}, if there exists a finite $F \subset G$ such 
that $G = FA$; 
\item[(c)] 
\emph{right piecewise syndetic}, or simply \emph{piecewise syndetic}, if there exists a 
finite $F \subset G$ such that $FA$ is thick. 
\end{enumerate} 
\end{definition}

Left thick, left syndetic, and left piecewise syndetic sets are defined by analogy;  
in what follows, we consider only right versions and omit the word ``right.''

\begin{definition}
Given a subset $A$ of a group $G$, we shall refer to the least cardinality of a set 
$F \subset G$ for which $G=FA$ as the \emph{syndeticity index}, or simply \emph{index} (by analogy 
with subgroups) of $A$ in $G$. Thus, a set is syndetic if and only if it is of finite index. 
We also define the \emph{thickness index} of $A$ as the least cardinality of $F\subset G$ for which 
$FA$ is thick. 
\end{definition}

A set $A\subset\mathbb Z$ is syndetic if and only if the gaps between neighboring elements of $A$ 
are bounded, and $B\subset \mathbb Z$ is thick if and only if it contains arbitrarily long 
intervals of consecutive integers. The intersection of any such sets $A$ and $B$ is piecewise 
syndetic; clearly, such a set is not necessarily syndetic or thick (although it may as well be both 
syndetic and thick). The simplest general example of a syndetic set in a group 
is a coset of a finite-index subgroup. 

In what follows, when dealing with general groups, we use multiplicative notation, 
and when dealing with Abelian ones, we use additive notation. 

Given a set $A$ in a group $G$, by $\langle A\rangle$ we denote the subgroup of $G$ generated 
by~$A$.

As mentioned, we are particularly interested in Boolean groups, i.e., groups in which all elements 
are self-inverse. All such groups are Abelian. 
Moreover, any Boolean group $G$ can be treated as a countable-dimensional 
vector space over the two-element field $\mathbb Z_2$; 
therefore, for some set $X$ (basis), $G$ can be represented 
as the free Boolean group $B(X)$ on $X$, i.e., as $[X]^{<\omega}$ with zero 
$\varnothing$, which we 
denote by $\zero$, and the 
operation of symmetric difference, or  Boolean sum, which we denote by $\sd$: 
$A\sd B = (A\cup B)\setminus A\cap B$. The elements of $B(X)$ (i.e., 
finite subsets of $X$) are called \emph{words}. The \emph{length} of a 
word equals its cardinality. The basis $X$ is embedded in $B(X)$ as the set of 
words of length~1.  Given $n\in \omega$, 
we use the standard notation $B_n(X)$ for the set of words of 
length at most $n$; thus, $B_0(X)=\{\zero\}$, $B_1(X) = X\cup \{0\}$, and 
$B(X)=\bigcup_{n\in \omega} B_n(X)$. For the set of words of length precisely $n$, 
where $n\in \mathbb N$, 
we use the notation $B_{=n}(X)$; we have $B_{=n}(X)=B_n(X)\setminus B_{n-1}(X)$. 

Any free filter $\mathscr F$ on an infinite set $X$ determines a topological 
space $X_{\mathscr F}=X\cup \{*\}$ with one nonisolated point $*$; the neighborhoods of this point 
are $A\cup \{*\}$ for $A\in \mathscr F$. The topology of the free Boolean topological group 
$B(X_{\mathscr F})=[X\cup \{*\}]^{<\omega}$ on this space, 
that is, the strongest group topology that induces 
the topology of $X_{\mathscr F}$ on $X\cup \{*\}$, is described in detail in~\cite{axioms}. One of 
the possible descriptions is as follows. For each $n\in \mathbb N$, we fix an arbitrary
sequence of neighborhoods~$V_n$ of $*$, that is, of $A_n\cup\{*\}$, where $A_n\in \mathcal F$, 
and set 
\begin{align*}
U(V_n) &= \{x\sd y\colon x,y\in V_n\}\qquad \text{for $n\in \mathbb N$}\\
\intertext{and}
U\bigl((V_n)_{n\in \mathbb N}\bigr) &=
\bigcup_{n\in \mathbb N} (U(V_1) \sd U(V_2) \sd\dots \sd U(V_n))\\
&= 
\bigcup_{n\in \mathbb N}\{x_1\sd y_1\sd  \dots \sd x_n\sd y_n : 
x_i,y_i\in A_i\text{ for }i\le n\}.
\end{align*}
In particular, the  subgroup generated by $(A\cup \{*\})\sd (A\cup \{*\})$ is a neighborhood 
of zero for any $A\in \mathscr F$.
Clearly, for  $n\in \omega$, 
a set $Y\subset B_{=2n}(X_{\mathscr F})$ is a trace on $B_{=2n}(X_{\mathscr F})$ 
of a neighborhood of zero in $B(X_{\mathscr F})$ 
if and only if it contains a set of the form 
$\bigl(\underbrace{(A\cup\{*\})\sd \dots\sd (A\cup\{*\})}_{\text{$2n$ times}}\bigr)
\cap B_{=2n}(X_{\mathscr F})=[A\cup\{*\}]^{2n}$, 
and a set\footnote{Recall that $X\subset X_{\mathscr F}=X\cup \{*\}$, and, therefore, 
$B(X)$ (without topology) is naturally embedded in $B(X_{\mathscr F})$ as a subgroup.} 
$Y\subset B_{=2n}(X) \subset B_{=2n}(X_{\mathscr F})$ 
is a trace on $B_{=2n}(X)$ of a neighborhood of zero in $B(X_{\mathscr F})$ 
if and only if it contains a set of the form 
$\bigl(\underbrace{A\sd \dots \sd A}_{\text{$2n$ times}}\bigr)\cap B_{=2n}(X)= 
[A]^{2n}$. The intersection of a neighborhood of zero with $B_{=k}(X_{\mathscr F})$ may be empty 
for all odd~$k$.

In what follows, we deal with rapid, $\kappa$-arrow, and Ramsey filters and ultrafilters. 

\begin{definition}[\cite{Mokobodzki}]
A filter $\mathscr F$ on $\omega$ is said to be \emph{rapid} if every function 
$\omega\to\omega$ is majorized by the increasing enumeration of some element of $\mathscr F$. 
\end{definition}

Clearly, any filter containing a rapid filter is rapid as well; thus, the existence of rapid 
filters is equivalent to that of rapid ultrafilters. Rapid ultrafilters are also known as 
semi-$Q$-point, or weak $Q$-point, ultrafilters. 
Both the existence and nonexistence of rapid ultrafilters is 
consistent with ZFC (see, e.g., \cite{Mathias} and \cite{Miller}). 

The notions  of $\kappa$-arrow and Ramsey filters are closely related to Ramsey theory, 
more specifically, to the notion of homogeneity with respect to a coloring, or partition. 
Given a set $X$ and positive integers $m$ and $n$, by an \emph{$m$-coloring} of $[X]^n$ we mean 
any map $c\colon X\to Y$ of $X$ to a set $Y$ of cardinality $m$. Any such coloring determines 
a partition of $X$ into $m$ disjoint pieces, each of which is assigned a color $y\in Y$. 
A set $A\subset X$ is said to be \emph{homogeneous} with respect to $c$, or \emph{$c$-homogeneous}, 
if $c$ is constant on $[A]^n$. The celebrated Ramsey theorem (finite version) asserts that, 
given any positive integers $k$, $l$, and $m$, there exists a positive integer $N$ such that, 
for any $k$-coloring $c\colon [X]^l\to Y$, where $|X|\ge N$ and $|Y|=k$, there exists a 
$c$-homogeneous set $A\subset X$ of size $m$.

We consider $\kappa$-arrow and Ramsey filters on any, not necessarily countable, infinite sets. 
For convenience, we require these filters to be uniform, i.e., nondegenerate in the sense that 
all of their elements have the same cardinality (equal to that of the underlying set). 

\begin{definition}
Let $\kappa$ be an infinite cardinal, and let $\mathscr F$ be a uniform filter
on a set $X$ of cardinality  $\kappa$. 
\begin{enumerate}[{\rm (i)}]
\item
We say that $\mathscr F$ is a \emph{Ramsey} filter if, for 
any 2-coloring $c\colon [X]^2\to \{0,1\}$, there exists a $c$-homogeneous 
set $A\in \mathscr U$. 
 \item
Given an arbitrary cardinal $\lambda\le \kappa$, we say that 
$\mathscr F$ is a \emph{$\lambda$-arrow} filter if, 
for any 2-coloring $c\colon [X]^2\to \{0, 1\}$, there exists 
either a set $A\in \mathscr F$ such that $c([A]^2)=\{0\}$ or a set  $S\subset X$ with 
$|S|\ge \lambda$ such that $c([S]^2)=\{1\}$.
\end{enumerate}
\end{definition} 

Any filter $\mathscr F$ on $X$ which is Ramsey or $\lambda$-arrow for $\lambda\ge 3$ 
is an ultrafilter. Indeed, let $S\subset X$ and consider the coloring 
$c\colon [X]^2 \to \{0, 1\}$ defined by
$$
c(\{x, y\})=\begin{cases}0&\text{if $x,y\in S$ or $x, y\in X\setminus S$,}\\ 
1&\text{otherwise}.
\end{cases}
$$
Clearly, any $c$-homogeneous set containing more than two points is contained entirely 
in $S$ or in $X \setminus S$; therefore, either $S$ or $X\setminus S$ 
belongs to $\mathscr F$, so that $\mathscr F$ is an ultrafilter.

According to Theorem~9.6 in~\cite{Comfort-Negrepontis}, 
if $\mathscr U$ is a Ramsey ultrafilter on $X$, then, for any $n < \omega$ and any 
2-coloring $c\colon [X]^n\to \{0,1\}$, there exists a $c$-homogeneous set 
$A\in \mathscr U$. 

It is easy to see that if $\mathscr F$ is $\lambda$-arrow, then, for any $A\in \mathscr F$ and 
any $c\colon [A]^2\to \{0, 1\}$, there exists either a set $B\in \mathscr F$ such that  
$B\subset A$ and $c([B]^2)=\{0\}$ or a set  $S\subset A$ with $|S|\ge \lambda$ 
such that $c([S]^2)=\{1\}$.

In~\cite{Baumgartner-Taylor}, where $k$-arrow ultrafilters for finite $k$ were introduced, 
it was shown that the existence of a $3$-arrow (ultra)filter on $\omega$ implies 
that of a $P$-point ultrafilter; therefore, the nonexistence of 
$\kappa$-arrow ultrafilters for any $\kappa\ge 3$ is consistent with ZFC~(see \cite{Shelah}).

On the other hand, the continuum hypothesis implies the 
existence of $k$-arrow ultrafilters on $\omega$ for any $k\le \omega$. To formulate a 
more delicate assumption under which $k$-arrow ultrafilters exist, we need more 
definitions. Given a uniform filter $\mathscr F$ on $\omega$, a set $B \subset \omega$ is called a 
\emph{pseudointersection} of $\mathscr F$ if the complement $A\setminus B$ is finite for all
$A \in \mathscr F$. The \emph{pseudointersection number} $\mathfrak p$ 
is the smallest size of a uniform filter on $\omega$ 
which has no infinite pseudointersection. It is easy to 
show that $\omega_1\le \mathfrak p\le 2^\omega$, so that, 
under the continuum hypothesis, $\mathfrak p = 2^\omega$.  It is also consistent with ZFC 
that, for any regular cardinals $\kappa$ and $\lambda$ such that $\omega_1\le \kappa\le \lambda$, 
$2^\omega = \lambda$ and $\mathfrak p = \kappa$  (see~\cite[Theorem~5.1]{Handbook}). 
It was proved in~\cite{Baumgartner-Taylor} that, under the assumption $\mathfrak p =2^\omega$ 
 (which is referred to as P(c) in \cite{Baumgartner-Taylor}), there exist $\kappa$-arrow 
ultrafilters on $\omega$ for all $\kappa\le \omega$. Moreover, for each $k\in \mathbb N$, 
there exists a $k$-arrow ultrafilter on $\omega$ 
which is not $(k+1)$-arrow, and there exists an ultrafilter which is $k$-arrow for each 
$k\in \mathbb N$ but is not Ramsey and hence not 
$\omega$-arrow~\cite[Theorems~2.1 and~4.10]{Baumgartner-Taylor}. 

In addition to the free group topology of Boolean groups on spaces generated by filters, 
we consider the \emph{Bohr topology} on arbitrary abstract and topological groups. 
This is the weakest group topology with respect to which 
all homomorphisms to compact topological groups 
are continuous, or the strongest totally bounded group topology; the Bohr topology on 
an abstract group (without topology) is defined as the Bohr topology 
on this group endowed with the discrete topology. 

Finally, we need the definition of a minimal dynamical system. 

\begin{definition}
Let $G$ be a monoid with identity element $e$. 
A pair $(X, (T_g)_{g\in G})$, where $X$ is a topological space and $(T_g)_{g\in G}$ is a family of 
continuous maps $X\to X$ such that $T_e$ is the identity map and $T_{gh}= T_g\circ T_h$ for any 
$g, h\in G$, is called a \emph{topological dynamical system}. Such a system is said 
to be \emph{minimal} if no proper closed subset of $X$ is $T_g$-invariant for all $g\in G$. 
\end{definition}

We sometimes identify sequences with their ranges. 

All groups considered in this paper are assumed to be infinite, and all filters are assumed 
to have empty intersection, i.e., to contain the Fr\'echet filter 
of all cofinite subsets (and hence be free).  

\section{Properties of Large Sets}
\label{ss1}
We begin with well-known general properties of large sets defined above. Let $G$ be a group.  

\textbf{1.~}A set $A\subset G$ is thick if and only if the family $\{g A: g\in G\}$ 
of all translates of 
$A$ has the finite intersection property. 

Indeed, this property means that, for every finite subset 
$F$ of $G$, there exists an $h\in \bigcap_{g\in F} g^{-1} A$, and this, in turn, means that 
$gh \in A$ for each $g\in F$, i.e., $Fh \subset A$. 

\smallskip

\textbf{2~\cite[Theorem~2.4]{BHM}.}\enspace A set $A$ is syndetic 
if and only if $A$ intersects every thick set nontrivially, or, equivalently, if its complement 
$G\setminus A$ is not thick.

\smallskip

\textbf{3.}\enspace A set $A$ is thick 
if and only if $A$ intersects every syndetic set nontrivially, or, equivalently, if its 
complement $G\setminus A$ is not syndetic. 

\textbf{4~\cite[Theorem~2.4]{BHM}.}\enspace 
A set $A$ is piecewise syndetic if and only if there exists a syndetic set $B$
and a thick set $C$ such that $A = B \cap C$.

\smallskip

\textbf{5~\cite[Theorem~4.48]{Hindman-Strauss}.}\enspace
A set $A$ is thick if and only if 
$$
\overline A^{\beta G} = \{p\in \beta G: A\in p\}
$$
(the closure of $A$ in the Stone--\v Cech compactification $\beta G$ of $G$ with the discrete 
topology) contains a left ideal of the semigroup $\beta G$.

\textbf{6~\cite[Theorem~4.48]{Hindman-Strauss}.}\enspace 
A set $A$ is syndetic if and only if every left ideal of $\beta G$ intersects $\overline A^{\beta 
G}$. 

\textbf{7.}\enspace The families of thick, syndetic, and piecewise syndetic sets are closed with 
respect to taking supersets. 

\textbf{8.}\enspace Thickness, syndeticity, and piecewise syndeticity are translation invariant. 

\textbf{9~\cite[Theorem~2.5]{BHM}.}\enspace Piecewise syndeticity 
is \emph{partition regular}, i.e., whenever a 
piecewise syndetic set is partitioned into finitely many subsets, one of these subsets is piecewise 
syndetic. 

\textbf{10~\cite[Theorem~2.4]{BHM}.}\enspace For any thick set $A\subset G$, there exists 
an infinite sequence $B=(b_n)_{n\in \mathbb N}$ in $G$ such that 
$$ 
\operatorname{FP}(B)=\{x_{n_1}x_{n_2}\dots x_{n_k}:k, n_1, n_2, \dots, n_k\in \mathbb N,\ 
n_1<n_2<\dots < n_k\}
$$
is contained in $A$.

\textbf{11.}\enspace Any $\IP^*$-set in $G$, i.e., a set intersecting any infinite set of the 
form $\operatorname{FP}(B)$, is syndetic. This immediately follows from properties~2 and~10. 

\section{Fat Sets}
\label{ss2}

As mentioned at the beginning of this section, in~\cite{arXiv}, Reznichenko and the author 
introduced a new\footnote{Later, we have found out that 
similar subsets of $\mathbb Z$ had already been used in~\cite{BFW}: 
the $\Delta^*_n$-sets considered there and $n$-fat 
subsets of $\mathbb Z$ are very much alike.} class of large sets, which we called fat; they have played the key role in our construction of 
nonclosed discrete subsets in topological groups. 

\begin{definition}
We say that a subset $A$ of a group $G$ is \emph{fat} in $G$ 
if there exists a positive integer $m$ such that any $m$-element set $F$ in  
$G$ contains a two-element subset $D$ for which $D^{-1}D\subset A$. 
The least number $m$ with this property is called the \emph{fatness} of $A$. 

We shall refer to fat sets of fatness $m$ as \emph{$m$-fat} sets.
\end{definition}

In a similar manner, $\kappa$-fat sets for any cardinal $\kappa$ can defined. 

\begin{definition}
Given a cardinal $\kappa$, we say that a subset $A$ of a group $G$ is \emph{$\kappa$-fat} in $G$ 
if any set $S\subset G$ with $|S|=\kappa$  contains a two-element subset $D$ for which $D^{-1}D\subset A$. 
\end{definition}

The notions of an $\omega$-fat and a $k$-fat set are very similar to but 
different from those of $\Delta^*$- and $\Delta_k^*$-sets. $\Delta^*$-Sets 
were introduced and studied in \cite{BHM} for  arbitrary semigroups, and 
$\Delta^*_k$-sets with $k\in \mathbb N$ were  defined in~\cite{BFW}
for the case of $\mathbb Z$. 

\begin{definition}
Given a finite of countable cardinal $\kappa$ and a sequence $(g_n)_{n\in\kappa}$ 
in a group $G$, we set 
$$
\Delta\bigl((g_n)_{n\in\kappa}\bigr) = \{x \in G : \text{there exist $m < n<\kappa$ 
such that $x = g_m^{-1}g_n$}\}
$$  
and 
$$
\Delta_D \bigl((g_n)_{n\in\kappa}\bigr) = 
\{x \in G: \text{there exist $m < n<\kappa$ such that $x 
= g_ng_m^{-1}$}\}. 
$$

A subset of a group $G$ is called a \emph{right} (\emph{left}) \emph{$\Delta_{\kappa}^*$-set} 
if it intersects $\Delta_I\bigl((g_n)_{n\in\kappa}\bigr)$ (respectively, 
$\Delta_D\bigl((g_n)_{n\in\kappa}\bigr)$) for any one-to-one sequence 
$(g_n)_{n\in\kappa}$ in $G$. $\Delta^*_\omega$-sets are 
referred to as \emph{$\Delta^*$-sets}. 
\end{definition}

\begin{remark*}
For any one-to-one sequence $S=(g_n)_{n\in \kappa}$ in a Boolean group with zero 
$\zero$, we have $\Delta_I(S) = \Delta_D(S) = (S\sd S)\setminus \{\zero\}$. 
Hence any $\kappa$-fat set in such a group is a right and left $\Delta_{\kappa}^*$-set. 
Moreover, the only difference between 
$\Delta_{\kappa}^*$- and $\kappa$-fat sets 
in a Boolean group is in that the latter must contain~$\zero$.  
\end{remark*}

The most obvious feature distinguishing fatness among other notions of largeness is symmetry 
(fatness has no natural right and left versions). In return, translation invariance is sacrificed. 
Thus, in studying fat sets, it makes sense to consider also their translates. 

Clearly, a $2$-fat set in a group must coincide with this group. The simplest nontrivial example 
of a fat set is a subgroup of finite index $n$; its fatness equals $n+1$ (any $(n+1)$-element 
subset has two elements $x$ and $y$ in the same coset, and both $x^{-1}y$ and $y^{-1}x$ belong 
to the subgroup).  

It seems natural to refine the definition of fat sets by requiring $A \cap F^{-1}F$ to be of 
prescribed size rather than merely nontrivial. However, this (and even 
a formally stronger) requirement does not introduce anything new. 

\begin{proposition}[{\cite[Proposition~1.1]{arXiv}}]
\label{fat Ramsey}
For any fat set $A$ in a group $G$  and any  positive 
integer~$n$, there exists a positive integer $m$ such that any $m$-element set $F$ in $G$ 
contains an $n$-element subset $F'$ for which 
$F'^{-1}F'\subset A$. 
\end{proposition}

Indeed, considering the coloring $c\colon [G]^2\to \{0, 1\}$ defined by 
$c(\{x,y\})=1 \iff x^{-1}y, y^{-1}x\in A$ and applying the finite Ramsey theorem, we find 
a $c$-homogeneous set $F$ of size $n$ (provided that $m$ is large enough). If $n$ is no smaller 
than the fatness of $A$ (which we can assume without loss of generality), then $c([F]^2)=\{1\}$.

There is yet another important distinguishing feature of fat sets, namely, the finite intersection 
property. Neither thick, syndetic, nor piecewise syndetic sets have this property. 
(Indeed, the disjoint sets of even and odd numbers are syndetic in $\mathbb Z$, 
and $\bigcup_{i\ge0} [2^{2i}, 2^{2i+1}) \cap \mathbb Z$ and $\bigcup_{i\ge 1} [2^{2i-1},2^i)$ 
are thick.) The following theorem is valid. 

\begin{theorem}[\cite{arXiv}]  
\label{fat filter}
Let $G$ be a group. 
\begin{enumerate}[\rm (i)]
\item
If $A\subset G$ is fat, then so is $A^{-1}$\textup.
\item
If $A\subset B\subset G$ and $A$ is fat, then so is $B$\textup.
\item
If $A\subset G$ and $B\subset G$ are fat, then so is $A\cap B$. 
\end{enumerate} 
\end{theorem}

Assertions (i) and (ii) are obvious, and (iii) follows from Proposition~\ref{fat Ramsey}.

\begin{proposition}
\label{3-fat cover}
If $G$ is a group, $S\subset G$, and $S\cap (SS\cup S^{-1}S^{-1})=\varnothing$, 
then $G\setminus S$ is $3$-fat. 
\end{proposition} 

\begin{proof}
Take any three different elements $a, b, c\in G$. We must show that the identity element $e$ 
belongs to $G\setminus S$ (which is true by assumption) and  either 
$(a^{-1}b)^{\pm1}\in G\setminus S$, 
$(b^{-1}c)^{\pm1}\in G\setminus S$, or $(c^{-1}a)^{\pm1}\in G\setminus S$. 
Assume that, on the contrary, $(a^{-1}b)^{\varepsilon}\in S$ (i.e., 
$a^{-1}b\in S^{\varepsilon}$),
$b^{-1}c\in S^{\delta}$, and $c^{-1}a\in S^{\gamma}$ for some $\epsilon, \delta, \gamma\in 
\{-1, 1\}$. At least two of the three numbers $\varepsilon$, $\delta$, and $\gamma$ are equal. 
Suppose for definiteness that $\varepsilon= \delta$. Then we have 
$c^{-1}a = c^{-1}b b^{-1}a\in S^{-\varepsilon}S^{-\varepsilon}$, which contradicts the assumption 
$S\cap (S^2\cup S^{-2})=\varnothing$.
\end{proof}

We see that the family of fat sets in a group resembles, 
in some respects, a base of neighborhoods of the identity element for a group topology. 
However, as we  shall see in the next section, it does not generate a group topology 
even in a Boolean group: any Boolean group has a 3-fat subset $A$ containing no set 
of the form $B\sd B$ for fat $B$. On the other hand, 
very many groups admit of group topologies in which all neighborhoods of the identity 
element are fat; for example, such are topologies generated by normal subgroups of finite index. 
A more precise statement is given in the next section. Before turning to related questions, we 
consider how fat sets fit into the company of other large sets. 

We begin with a comparison of fat and syndetic sets.

\begin{proposition}[{see~\cite[Proposition~1.7]{arXiv}}]
\label{Prop1}
Let $G$ be any group with identity element~$e$.
Any fat set $A$ in $G$ is syndetic, and its syndeticity index is less than its fatness. 
\end{proposition}

\begin{proof}
Let $n$ denote the fatness of $A$. Take a finite set $F\subset G$ with $|F|=n-1$ such that 
$x^{-1} y\notin A$ or $y^{-1} x\notin A$ for any different $x, y \in F$. 
Pick any $g\in G\setminus F$.  
Since $|F\cup\{g\}|=n$, it follows that $x^{-1}g\in A$ and $g^{-1}x\in A$ for some $x\in F$, whence 
$g\in xA$, i.e., $G\setminus F\subset FA$. By definition, the identity element of $G$ belongs to 
$A$, and we finally obtain 
$G = FA$. 
\end{proof}

Examples of nonfat syndetic sets are easy to construct: any coset of a finite-index subgroup 
in a group is syndetic, while only one of them (the subgroup itself) is fat. However, the existence 
of syndetic sets with nonfat translates is not so obvious. An example of  such a set in $\mathbb Z$ 
can be extracted from~\cite{BFW}. 

\begin{example}
\label{example2}
There exists a syndetic set in $\mathbb Z$ such that none of its translates is fat. 
This is, e.g., the set constructed in~\cite[Theorem~4.3]{BFW}. Namely, let $C=\{0,1\}^{\mathbb Z}$, 
and let $\tau\colon C\to C$ be the shift, i.e., the map defined by $\tau(f)(n)=f(n+1)$ for $f\in C$. It 
was proved in~\cite[Theorem~4.3]{BFW} that if $M\subset C$ is a  
minimal closed $\tau$-invariant 
subset\footnote{Then the support of each $f\in M$ is syndetic in $\mathbb Z$ 
(see, e.g.,~\cite{Furstenberg}).} and the dynamical system 
$(M, (\tau^n)_{n\in \mathbb Z})$ satisfies a certain condition\footnote{Namely, 
is weakly mixing; see, e.g.,~\cite{Furstenberg}}, then the support of any $f\in M$ 
is syndetic but  not piecewise Bohr; the latter means that it cannot be represented 
as the intersection of a thick set and a set 
having nonempty interior in the Bohr topology on $\mathbb Z$. Clearly, any translate of 
$\operatorname{supp}f$ has these properties as well.  On the other hand, according to Theorem~II
in~\cite{BFW}, any $\Delta^*_n$-set in $\mathbb Z$ (i.e., any set intersecting 
the set of differences $\{k_j-k_i:i<j\le n\}$ for each $n$-tuple $(k_1, \dots, k_n)$ of different 
integers) is piecewise Bohr. Since every $n$-fat set is a $\Delta^*_n$-set, it follows that the 
translates of $\operatorname{supp}f$ cannot be fat.
\end{example}

Bearing in mind our particular interest in Boolean groups, we also give a similar example for 
a Boolean group. 

\begin{example}
\label{example3}
We construct a syndetic set in the Boolean group $B(\mathbb Z)$ with nonfat translates. 
Let $S$ be a syndetic set in $\mathbb Z$ all of whose translates are not $\Delta^*_n$-sets for 
all~$n$ (see Example~\ref{example2}).  
By definition, $\mathbb Z = \bigcup_{k\le r}(s_k+S)$ for some $r\in \mathbb N$ and different 
$s_1, \dots, s_r \in \mathbb Z$.  We set 
\begin{multline*}
S'_k =\{x_1\sd \dots \sd x_n: \text{$n\in \mathbb N$, $x_i\in \mathbb Z$  for $i\le n$, 
$x_i\ne x_j$ for $i\ne j$,}\\ 
 \text{$\{x_1, \dots, x_n\}\cap \{s_1, \dots, s_r\}=\{s_k\}$, 
$\sum_{i\le n} x_i\in 2s_k+S$}\}, \qquad k\le r,
\end{multline*}
and
$$
S'=\bigcup_{k\le r}S'_k.
$$
We have 
\begin{multline*}
s_k\sd S'_k =\{x_1\sd \dots \sd x_n: \text{$n\in \mathbb N$, $x_i\in \mathbb Z$  for $i\le n$, 
$x_i\ne x_j$ for $i\ne j$,}\\ 
 \text{$\{x_1, \dots, x_n\}\cap \{s_1, \dots, s_r\}=\varnothing$, 
$\sum_{i\le n} x_i\in s_k+S$}\}, \qquad k\le r.
\end{multline*}
Since $\bigcup_{k\le r}(s_k+S)=\mathbb Z$, it follows that 
\begin{multline*}
\bigcup_{k\le r}(s_k\sd S')\\
 \subset 
\{x_1\sd \dots \sd x_n: \text{$n\in \mathbb N$, $x_i\in \mathbb Z$ for $i\le n$, 
$\{x_1, \dots, x_n\}\cap \{s_1, \dots, s_r\}=\varnothing$}\}.
\end{multline*}
Obviously, the set on the right-hand side of this inclusion is syndetic; therefore, so is $S'$. 

Let us show that no translate of $S'$ is fat. 
Suppose that, on the contrary, 
$k,n\in \mathbb N$, $z_1, \dots, z_k\in \mathbb Z$, $w=z_1\sd \dots \sd z_k$, 
and $w\sd S'$ is $n$-fat.
Take any different $k_1, \dots, k_n\in \mathbb Z$ larger than the absolute values of all elements of 
$w$ (which is a finite subset of $\mathbb Z$) and of all $s_i$, $i\le r$. We set 
\begin{multline*}
F=\{k_1, k_1\sd(-k_1)\sd k_2, k_1\sd(-k_1)\sd k_2\sd (-k_2)\sd k_3, \\
\dots, k_1\sd(-k_1)\sd k_2\sd (-k_2)\sd\dots \sd k_{n-1}\sd (-k_{n-1})\sd k_n\}. 
\end{multline*}
Suppose 
that there exist different $x,y\in F$ for which $x\sd y\in w\sd S'$, i.e., there exist 
$i,j\le n$ for which $i<j$ and 
\begin{multline*}
k_1\sd(-k_1)\sd\dots \sd k_{i-1}\sd (-k_{i-1})\sd k_i 
\sd
k_1\sd(-k_1)\sd\dots \sd k_{j-1}\sd (-k_{j-1})\sd k_j \\
= 
k_i\sd  k_i\sd(-k_i)\sd k_{i+1}\sd (-k_{i+1})\sd\dots\sd k_{j-1}\sd (-k_{j-1})\sd k_j 
=w\sd s\in w\sd S',
\end{multline*}
where $s$ is an element of $S'$  
and hence belongs to $S'_l$ for some $l\le r$, which means, in particular, that 
$s$ contains precisely one of the letters $s_1, \dots, s_r$, namely, $s_l$. 
There are no such letters among $\pm k_i, \dots, \pm k_{j-1}, k_j$. Therefore, one of the letters 
$z_m$ (say $z_1$) is $s_l$. The other letters of $w$ do not equal 
$\pm k_i, \dots, \pm k_{j-1}, k_j$ either and, therefore, are canceled with letters of $s \in S'$ 
in the word $w+s$. By the definition of the set $S'$ containing $s$, one letter of the 
word $w$ (namely, $z_1=s_l$) 
belongs to the set $\{s_1, \dots, s_r\}$ and the other letters do not. 
Since the sum (in $\mathbb Z$) of the integer-letters of $s$ belongs to $2s_l + S$ (by the definition 
of $S'_l$) and $s_l=z_1$, it follows that the sum of letters of $w+s$ belongs 
to $S+z_1-z_2-\dots -z_k$ and the letter $z_1$ is determined uniquely for the given word $w$. 
To obtain a contradiction, it remains to recall that the translates of $S$ 
(in particular, $S+z_1-z_2-\dots -z_k$) are not $\Delta^*_n$-sets in $\mathbb Z$ and choose 
$k_1, \dots , k_n$ so that $\{k_j-k_i:i<j\le n\}\cap (S+z_1-z_2-\dots -z_k)=\varnothing$. 
\end{example}

\begin{example}
There exist fat sets which are not thick and thick sets which are not 
fat. Indeed, as mentioned, any proper finite-index group is fat, but it cannot be thick by the 
first property in the list of properties of large sets given above. 

An example of a nonfat thick set is, e.g., any thick nonsyndetic set.  
In an infinite Boolean group $G$, such a set can be constructed as follows. 
Take any basis $X$ in $G$ (so 
that $G=B(X)$), fix any nonsyndetic thick set $T$ in $\mathbb N$ (say 
$T=\bigcup_n([a_n, b_n]\cap \mathbb N)$, where the $a_n$ and $b_n$ are numbers such that 
the $b_n-a_n$ and the $a_{n+1}-b_n$ increase without bound), and consider the set 
$$
A=\{x_1\sd \dots \sd x_n\in B(X): 
n\in T, x_i\in X \text{ for }i\le n, x_i\ne x_j \text{ for }i\ne j\}
$$ 
of all words in $B(X)$ whose lengths belong to $T$.  
The thickness of this set is obvious (by the same property~1), because the translate of 
$A$ by any word $g\in B(X)$ of any length $l$ surely contains all words whose lengths belong to 
$\bigcup_n([a_n+l, b_n-l]\cap \mathbb N)\subset T$ and, therefore, intersects $A$. 
However, $A$ is not fat, because it misses all words whose lengths belong to the set 
$\bigcup_n((b_n, a_{n+l})\cap \mathbb N)$. The last set contains at least one even positive 
integer $2k$. It remains to choose different points $x_1, x_2, \dots$ in $X$, set 
$B= \{x_{kn+1}\sd x_{kn+2}\dots \sd x_{kn+k}: n\in \omega\}$, and note that all nonempty words 
in $B\sd B$ have length $2k$. Therefore, 
$A$ is disjoint from $B\sd B$ (much more from $F\sd F$ for any finite $F\subset B$).  Note that 
the  translates of $A$ are not fat either, because both thickness and (non)syndeticity are 
translation invariant.
\end{example}

\begin{proposition}
Let $G$ be any group with identity element $e$.
\begin{enumerate}[{\rm (i)}]
\item
If a set $A$ in $G$ is $3$-fat, then $(G\setminus A)^{-1}(G\setminus A)\subset A$. 
\item 
If a set $A$ in $G$ is $3$-fat, then either $AA^{-1} = G$ or $A$ is a subgroup of index~$2$. 
\end{enumerate}
\end{proposition}

\begin{proof}
(i)\enspace
Suppose that $A$ is a $3$-fat subset of a group $G$ with identity element $e$. Take any different 
$x, y\notin A$ (if there exist no such elements, then there is nothing to prove). 
By definition, the set $\{x, y, e\}$ contains a two-element subset $D$ for which 
$D^{-1}D\subset A$. Clearly, $D\ne \{x,e\}$ and $D\ne \{y, e\}$. Therefore, $x^{-1}y\in A$ and 
$y^{-1}x\in A$ (and $e\in A$, too), whence $(G\setminus A)^{-1}(G\setminus A)\subset A$. 

(ii)\enspace
If $AA^{-1}\ne G$, then there exists a $g\in G$ for which $gA\cap A=\varnothing$. 
If $A$ is, in addition, $3$-fat, then (ii) implies $(gA)^{-1}gA = 
A^{-1}A\subset A$, which means that $A$ is a subgroup of $G$. According to (i), $A$ 
is syndetic of index at most~2; in fact, its index is precisely 2, because $A$ 
does not coincide with~$G$.
\end{proof}

\section{Quotient sets}
\label{ss3}

In~\cite{BHM} sets of the form $AA^{-1}$ or $A^{-1}A$ were naturally called \emph{quotient sets}. 
We shall refer to the former as \emph{right quotient sets} and to the latter as 
\emph{left quotient sets}. 
Thus, a set in a group $G$ is $m$-fat if it intersects nontrivially 
the left quotient set of any $m$-element subset of $G$. Quotient sets play a very important 
role in combinatorics, and their interplay with large sets is quite amazing. 

First, the passage to right quotient sets annihilates the difference between syndetic and piecewise 
syndetic sets.

\begin{theorem}[{see \cite[Theorem~3.9]{BHM}}] 
\label{BHM 3.9}
For each piecewise syndetic subset $A$ of a group $G$,
there exists a syndetic subset $B$ of $G$ such that $BB^{-1} \subset  
AA^{-1}$ and the syndeticity index of $B$ does not exceed the thickness index of $A$.
\end{theorem}

Briefly, the construction of $B$ given in~\cite{BHM} is as follows: we take a finite set $T$  
such that $TA$ is thick and, for each finite $F\subset G$, let 
$
\Phi_F= \{\varphi\in T^G: \bigcap_{x\in F} x^{-1}\varphi(x)A\ne \varnothing\}$. 
Then we  pick 
$\varphi^*$ in the intersection of all $\Phi_F$ 
(which exists since the product space $T^G$ is compact)  
and let $B=\{\varphi^*(x)^{-1}x: x\in G\}$. Since 
$\varphi^*(G)\subset T$, it follows that $TB=G$, which means that $B$ is syndetic and its index 
does not exceed $|T|=t$. Moreover, for any finite $F\subset B$, 
there exists a $g\in G$ such that 
$Fg\subset A$, and this implies $BB^{-1} \subset  
AA^{-1}$. 

In Theorem~\ref{BHM 3.9}, right quotient sets cannot be replaced by left ones: 
there are examples of piecewise syndetic sets $A$ such that $A^{-1}A$ does not contain $B^{-1}B$ 
for any syndetic $B$. One of such examples is provided by the following theorem. 

\begin{theorem}
\label{syndetic-fat}
\begin{enumerate}[{\rm (i)}]
\item
If a subset $A$ of a group $G$ is syndetic of index $s$,
then $A^{-1}A$ is fat, and its fatness does not exceed $s+1$. 
\item
If a subset $A$ of an Abelian group $G$ is piecewise syndetic of thickness index $t$,
then $A-A$ is fat, and its fatness does not exceed $t+1$. 
\item 
There exists a group $G$ and a thick (in particular, piecewise syndetic) set $A\subset G$ such that 
$A^{-1}A$ is not fat and, therefore, does not contain $B^{-1}B$ for any syndetic set.
\item 
If a subset $A$ of a group $G$ is thick, then $AA^{-1}=G$.
\end{enumerate}
\end{theorem}

\begin{proof}
(i)\enspace Suppose that $FA = G$, 
where $F=\{g_1, \dots, g_s\}$. Any $(s+1)$-element subset of $G$ has at least two points $x$ and $y$ 
in the same ``coset'' $g_iA$. We have $x= g_i a'$ and $y = g_ia''$, where $a', a''\in A$. Thus, 
$x^{-1}y, y^{-1}x\in A^{-1}A$. 

Assertion~(ii) follows immediately from (i) and Theorem~\ref{BHM 3.9}. 

Let us prove~(iii). Consider the  free group $G$ on two generators $a$ and $b$ and let 
$A$ be the set  of all words in $G$ whose last letter is $a$.  Then $A$ is thick (given any 
finite $F\subset G$, we have $Fa^n \subset A$ for sufficiently large $n$). Clearly, all nonidentity 
words in $A^{-1}A$ contain $a$ or $a^{-1}$. Therefore, if $F\subset G$ consists of words of the 
form $b^n$, then the intersection $F^{-1}F\cap A^{-1}A$ is trivial, so that $A^{-1}A$ is not fat.

Finally, to prove~(iv), take any $g\in G$. 
We have $A\cap gA\ne \varnothing$ (by property~1 in our list of properties of 
large sets). This means that $g\in AA^{-1}$.
\end{proof}

We see that the right quotient sets $AA^{-1}$ of thick sets $A$ are utmostly fat, 
while the left quotient sets $A^{-1}A$ may be rather slim. In the Abelian case, 
the difference sets of all thick sets coincide with the whole group.

It is natural to ask whether condition (i) in Theorem~\ref{syndetic-fat} characterizes fat sets in 
groups. In other words, given any fat set $A$ in a group, does there exist a syndetic 
(or, equivalently, piecewise syndetic) set $B$ 
such that $B^{-1}B\subset A$ (or $BB^{-1}\subset A$)? The answer is no, 
even for thick 3-fat sets in Boolean groups. The idea of 
the following example was suggested by arguments in paper~\cite{BFW} and 
in John Griesmer's note~\cite{Griesmer}, where the group~$\mathbb Z$ was considered. 

\begin{example}
\label{example}
Let $G$ be a countable Boolean group with zero $\mathbf{0}$. Any such group 
can be treated as the free Boolean group on $\mathbb Z$. 
We set 
$$
A=G\setminus \{m\sd n=\{m,n\}: m, n\in \mathbb Z, m<n, n-m= k^3 \text{  
for some }k\in \mathbb N\}.
$$
Clearly, $A$ is thick (if $F\subset G$ is finite and a word $g\in G$ is sufficiently long, then 
all words in the set $F\sd g$ have more than two letters and, therefore, belong to $A$).  
Let us prove that $A$ is 3-fat. Take any different $a, b, c\in G$. We must show that $a\sd b\in A$, 
$b\sd c\in A$, or $a\sd c\in A$. We can assume that $c=\mathbf{0}$; otherwise, we translate 
$a$, $b$, and $c$ by $c$, which does not affect the Boolean sums. 
Thus, it suffices to show that, given 
any different nonzero $x, y\not\in G$, we have $x\sd y\in A$. The condition 
$x, y\not\in G$ means that 
$x = \{k, l\}$, where $k<l$ and $l-k = r^3$ for some $r\in \mathbb Z$, and 
$y = \{m, n\}$, where $m<n$ and $n-m = s^3$ for some $s\in \mathbb Z$. Suppose for definiteness 
that $n> l$ or $n=l$ and $m> k$. 
If $x\sd y\not\in A$, then either $k=m$ and $l-n=t^3$ for some $t\in \mathbb N$, $l=m$ and 
$n-k = t^3$ for some $t\in \mathbb N$, or $l=n$ and $m-k = t^3$ for some $t\in \mathbb N$. In the 
first case, we have $l-k= l-n+n-m$, i.e., $r^3=t^3+s^3$; in the second, we have 
$n-k=n-m+l-k$, i.e., $t^3= s^3+r^3$; and in the third, we have $l-k=n-m+m-k$, i.e., $r^3=s^3+t^3$. 
In any case, we obtain a contradiction with Fermat's theorem. 

It remains to prove that there exists no syndetic 
(and hence no piecewise syndetic) $B\subset G$ for which $B\sd B\subset A$. 
Consider any syndetic set $B$. Let $F=\{f_1, \dots, f_k\}\subset G$ 
be a finite set for which $FB=G$, and let $m$ be the maximum  absolute value 
of all letters of words in $F$ (recall that all letters are integers). 
To each $n\in \mathbb Z$ with $|n|>m$  we assign a word $f_i\in F$ for which 
$n\in f_i\sd B$; 
if there are several such words, then we choose any of them. Thereby, we divide 
the set of all integers with absolute value larger than $m$ into $k$ pieces $I_1, \dots, I_k$. 
To accomplish our goal, it suffices to show that there is a piece $I_i$ containing two integers 
$r$ and $s$ such that $r-s=z^3$ for some $z\in \mathbb Z$. 
Indeed, in this case, we have $r\in f_i\sd B$ and 
$s\in  f_i\sd B$, so that $r\sd s\in B\sd B$. On the other hand, $r\sd s\not\in A$. 

From now on, we treat the pieces $I_1, \dots, I_k$ as subsets of $\mathbb Z$. 
We have $\mathbb Z = \{-m, -m+1, \dots, 0, 1, \dots,  m\}\cup I_1\cup \dots\cup I_k$. 
Since piecewise syndeticity is partition regular (see 
property~9 of large sets), one of the sets $I_i$, say $I_l$, is piecewise syndetic. Therefore, by 
Theorem~5, $I_l-I_l\supset S-S$ for some syndetic set $S\subset \mathbb Z$. 

Let $d^*(S)$ denote the upper Banach density of $S$, i.e., 
$$
d^*(S)=\lim_{n\to\infty}\limsup_{d\to\infty}\frac{|S\cap\{n,n+1,\cdots,n+d\}|}{d}.
$$
The syndeticity of $S$ in $\mathbb Z$ implies the existence of an $N\in \mathbb N$ such 
that every interval of integers longer than $N$  intersects $S$. Clearly,  we have 
$d^*(S) \ge 1/N$. Proposition~3.19 in~\cite{Furstenberg} asserts that 
if $X$ is a set in $\mathbb Z$ of positive upper 
Banach density and $p(t)$ is a polynomial taking on integer values at the integers and including 
$0$ in its range on the integers, then there exist  $x, 
y\in X$, $x\ne y$, and $z\in \mathbb Z$ such that $x-y= p(z)$ 
(as mentioned in~\cite{Furstenberg}, this was proved independently by S\'ark\"ozy). 
Thus, there exist 
different $x, y\in S$ and a $z\in \mathbb Z$ for which 
$x-y=z^3$. Since $S-S\subset I_l-I_l$, it follows that $z^3 = r-s$ for some $r, s\in I_l$, as 
desired.
\end{example}

\section{Large Sets and Topology}
\label{ss4}

In the context of topological groups quotient sets arise again, because for each neighborhood 
$U$ of the identity element, there must exist a neighborhood $V$ such that $V^{-1}V\subset U$ and 
$VV^{-1}\subset U$. Thus, if we know that a group topology consists of piecewise  syndetic sets, 
then, in view of Theorem~\ref{BHM 3.9}, we can assert that all open sets are syndetic, 
and so on. Example~\ref{example} shows that if $G$ is any countable Boolean topological 
group and all 3-fat sets are open in $G$, then some nonempty open sets in this group are 
not piecewise syndetic. Thus, all syndetic or piecewise syndetic subsets of a group $G$ 
do not generally form a group topology. Even their quotient (difference in the Abelian case) 
sets are insufficient; however, it is known that double difference sets of syndetic 
(and hence piecewise syndetic) sets in Abelian groups are neighborhoods of zero 
in the Bohr topology.\footnote{It follows, 
in particular, that, given any piecewise syndetic set $A$ in an Abelian group, 
there exists an infinite sequence of fat sets $A_1$, $A_2$, \dots such that $A_1-A_1\subset A+A-A-A$ 
and $A_{n+1}-A_{n+1}\subset A_n$ for all $n$ (because all Bohr open sets are syndetic).} 
These and many other interesting results concerning a relationship between 
Bohr open and large subsets of abstract and topological groups can be found 
in~\cite{Ellis, sumsets}. As to group topologies in which all open sets are large, 
the situation is very simple. 

\begin{theorem}
For any topological group $G$ with identity element $e$, the following conditions are equivalent:
\begin{enumerate}[{\rm (i)}]
\item 
all neighborhoods of $e$ in $G$ are piecewise syndetic;
\item 
all open sets in $G$ are piecewise syndetic;
\item 
all neighborhoods of $e$ in $G$ are syndetic;
\item 
all open sets in $G$ are syndetic;
\item 
all neighborhoods of $e$ in $G$ are fat;
\item 
$G$ is totally bounded.
\end{enumerate}
\end{theorem}

\begin{proof}
The equivalences (i)~$\Leftrightarrow$~(ii) and (iii)~$\Leftrightarrow$~(iv) 
follow from the obvious translation invariance of 
piecewise syndeticity and syndeticity. Theorem~\ref{BHM 3.9} implies (i)~$\Leftrightarrow$\ (iii), 
Theorem~\ref{syndetic-fat}\,(i) implies (iii)~$\Rightarrow$~(v), and Proposition~\ref{Prop1} 
implies (v)~$\Rightarrow$~(iii). The implication (iii)~$\Rightarrow$~(i) is trivial. Finally, 
(vi)~$\Leftrightarrow$~(iii) by the definition of total boundedness.
\end{proof}

Thus, the Bohr topology on a (discrete) group is the strongest group topology 
in which all open sets are syndetic (or, equivalently, piecewise syndetic, or fat). 

For completeness, we also mention the following corollary of 
Theorem~\ref{syndetic-fat} and Theorem~3.12 in~\cite{BHM}, 
which relates fat sets to topological dynamics. 

\begin{corollary}
If $G$ is an Abelian group with zero $0$, 
$X$ is a compact Hausdorff space, and $(X, (T_g)_{g\in G})$ 
is a minimal dynamical system, then 
the set $\{g \in G : U \cap T_g^{-1}U \ne \varnothing\}$ 
is fat for every nonempty open subset $U$ of $X$.
\end{corollary}

\section{Fat and Discrete Sets in Topological Groups} 
\label{ss5}

As mentioned above, fat sets were introduced in~\cite{arXiv} to construct discrete sets 
in topological groups. Namely, given a countable topological group $G$ 
whose identity element $e$ has 
nonrapid filter $\mathscr F$  of neighborhoods, we can construct a discrete set with precisely 
one limit point in this group as follows. The nonrapidness of $\mathscr F$ means that, 
given any sequence 
$(m_n)_{n\in \mathbb N}$ of positive integers, there exist finite sets $F_n \subset G$, $n \in 
\mathbb N$, such that each neighborhood of $e$ intersects some $F_n$ in at least $m_n$ points 
(see \cite[Theorem~3\,(3)]{Miller}). 
Thus, if we have a decreasing sequence of closed $m_n$-fat sets $A_n$ in $G$ such that $\bigcap 
A_n=\{e\}$, then the set 
$$ 
D=\bigcup_{n\in \mathbb N}\{a^{-1}b: a\ne b, a,b\in F_n, a^{-1}b \in A_n\} 
$$
is discrete (because $e\notin D$ and each $g\in G\setminus \{e\}$ has a neighborhood of the 
form $G\setminus A_{n}$ which contains only finitely 
many elements of $D$), and $e$ is the only limit point of $D$ (because, given any neighborhood 
$U$ of $e$, we can take a neighborhood $V$ such that $V^{-1}V\subset U$; we have $|V\cap F_n|\ge 
m_n$ for some $n$, and hence $(V\cap F_n)^{-1}(V\cap F_n)\cap A_n\ne \varnothing$, so that $U\cap 
D\ne \varnothing$). It remains to find a family of closed fat sets with trivial 
intersection and make it decreasing.  

The former task is easy to accomplish in any topological group: 
by Proposition~\ref{3-fat cover}, in any topological group $G$, the 
complements to open neighborhoods $gU$ of all $g\in G$ satisfying the condition $gU\cap 
(U^2 \cup (U^{-1})^2) =\varnothing$ form a family of closed $3$-fat sets with trivial intersection. 
In countable groups, this family can be made decreasing by using Theorem~\ref{fat filter}, according 
to which the family of fat sets has the finite intersection property. Unfortunately, no 
similar argument applies in the uncountable case, because countable intersections of fat sets 
may be very small. Thus, in $\mathbb Z_2^\omega$, the intersection of 
the 3-fat sets $H_n= \{f\in \mathbb Z_2^\omega: f(n)=0\}$ 
(each of which is a subgroup of index~2 open in the product topology) is trivial.

\section{Large Sets in Boolean Groups}
\label{ss6}
In the case of Boolean groups, many assertions concerning large sets can be refined. For example, 
properties~10 and~11 of large sets are stated as follows. 

\begin{proposition}
\label{prop Boolean1}
\begin{enumerate}[{\em (i)}]
\item
For any thick set $T$ in a Boolean group  $G$ with zero $0$, there exists an infinite subgroup 
$H$ of $G$ for which $T\cup \{0\} \supset H$. 
\item
Any set which intersects nontrivially all infinite subgroups in a Boolean group $G$ is syndetic.
\end{enumerate}
\end{proposition}

Note that this is not so in non-Boolean groups: the set  $\{n!: n\in \mathbb N\}$ intersects any 
infinite subgroup in $\mathbb Z$, but it is not syndetic, because the gaps between neighboring 
elements are not bounded. The complement of this set contains no infinite subgroups, and 
it is thick by property~2 of large sets. 

Another specific feature of thick sets in Boolean groups is given by the following proposition. 

\begin{proposition}
\label{prop Boolean2}
For any thick set $T$ in a countable Boolean group $G$ with zero $0$, 
there exists a set $A\subset G$ such that $T\cup \{0\} = A\sd A$ 
(and  $A\sd A\sd A\sd A = G$  by Theorem~\ref{syndetic-fat}\,\textup{(iv)}). 
\end{proposition}

Proposition~\ref{prop Boolean2} is an immediate corollary of Lemma~4.3 
in~\cite{BHM}, which says that any thick set in a countable Abelian group 
 equals $\Delta_I\bigl((g_n)_{n=1}^{\infty}\bigr)$ 
for some sequence $(g_n)_{n=1}^{\infty}$. 

In view of Example~\ref{example}, we cannot assert that the set $A$ in this proposition 
is large (in whatever sense), even for the largest (3-fat) nontrivial thick sets~$T$. 

The following statement can be considered as a partial analogue of Propositions~\ref{prop Boolean1} 
and~\ref{prop Boolean2} for $\Delta^*$- (in particular, fat) sets in Boolean groups. 

\begin{theorem}
For any $\Delta^*$-set $A$ in a Boolean group $G$ with zero $\mathbf 0$, 
there exists a $B\subset G$ with $|B|=|A|$ such 
that $B\sd B\subset A\cup\{\mathbf 0\}$. 
\end{theorem}

\begin{proof}
First, note that $|A|=|G|$. Any Boolean group is algebraically free; therefore, we can assume 
that $G=B(X)$ for a set $X$ with $|X|=|A|$. Let 
$$
A_2= A\cap B_{=2}(X)=\bigl\{\{x,y\}=x\sd y\in A: x,y\in X\bigr\}
$$ 
be the intersection of $A$ with the set of words of length~2. We have $|A_2|=|X|$, because 
$A$ must intersect nontrivially each countable set of the form $Y\sd Y$ for $Y\subset X$.  
Consider the coloring $c\colon [X]^2\to \{0, 1\}$ defined by 
$$
c(\{x,y\}) = \begin{cases} 0& \text{if $\{x,y\}\in A_2$},\\
1& \text{otherwise}.
\end{cases}
$$
According to the well-known Erd\"os--Dushnik--Miller theorem $\kappa\to(\kappa, \aleph_0)^2$ (see, 
e.g., \cite{Jech}), there exists either an infinite set $Y\subset X$ for which 
$[Y]^2\cap A_2=\varnothing$ or a set $Y\subset X$ of cardinality $|X|$ for which 
$[Y]^2\subset A_2$. The former case cannot occur, because $[Y]^2= Y\sd Y$ in $B(X)$, 
$[Y]^2\subset B_{=2}(X)$, and $A$ is a $\Delta^*$-set. Thus, the latter case occurs, and we set $B=Y$. 
\end{proof}

We have already distinguished between fat sets and translates of syndetic sets in Boolean 
groups (see Example~\ref{example3}). For completeness, we give the following example. 

\begin{example}
\label{example4}
The countable Boolean group $B(\mathbb Z)$ contains an $\IP^*$-set (see property~11 of large sets) 
which is not a $\Delta^*$-set. An example of such a set is constructed from the corresponding 
example in $\mathbb Z$ (see \cite[p.~177]{Furstenberg} in precisely the same way 
as Example~\ref{example3}.
\end{example}

\section{Large Sets in Free Boolean Topological Groups}
\label{ss7}

As shown in Section~\ref{ss4}, given any Boolean 
group $G$, the filter of fat sets in $G$ cannot be the filter of neighborhoods of zero 
for a group topology, because not all fat and even 3-fat sets are neighborhoods of zero 
in the Bohr topology. 
Moreover, if we fix any basis $X$ in $G$, so that $G=B(X)$, then not all traces of 3-fat sets on 
the set $B_{=2}(X)$ of two-letter words contain those of Bohr neighborhoods of zero. However, 
there are natural group topologies on $B(X)$ such that the topologies which they induce 
on $B_2(X)$ contain those generated by $n$-fat sets. 

\begin{theorem}
\label{B2(X)}
Let $k\in N$, and let $\mathscr F$ be a filter on an infinite set $X$. 
Then the following assertions hold.
\begin{enumerate}[{\rm (i)}]
\item 
For $k\ne 4$, the trace of any $k$-fat subset of $B(X)$
on $B_2(X)\subset B_2(X_{\mathscr F})$ contains that of a neighborhood of zero 
in the free group topology of $B(X_{\mathscr F})$ if and only if 
$\mathscr F$ is a $k$-arrow filter.
\item
If the trace of any $4$-fat set 
on $B_2(X)$ contains that of a neighborhood of zero 
in the free group topology of $B(X_{\mathscr F})$, then  
$\mathscr F$ is a $4$-arrow filter, and if $\mathscr F$ is a $4$-arrow filter, 
then the trace of any $3$-fat set 
on $B_2(X_{\mathscr F})$ contains that of a neighborhood of zero 
in the free group topology of $B(X_{\mathscr F})$. 
\item
The trace of any $\omega$-fat set 
on $B_2(X)$ contains that of a neighborhood of zero 
in the free group topology of $B(X_{\mathscr F})$ if and only if 
$\mathscr F$ is an $\omega$-arrow ultrafilter.
\end{enumerate} 
\end{theorem} 

The proof of this theorem uses the following lemma. 

\begin{lemma}
\label{2-words}
\begin{enumerate}[{\rm (i)}]
\item
If $k\ne 4$, $w_1, \dots, w_k\in B(X)$, and $w_i\sd w_j \in B_{=2}(X)$ for any $i< j\le k$, then there exist 
$x_1, \dots, x_k\in X$ such that $ w_i\sd w_j=x_i\sd x_j$ for any $i< j\le k$. 
\item
If $k=4$, $w_1, w_2, w_3, w_4\in B(X)$, and $w_i\sd w_j \in B_{=2}(X)$ for any $i< j\le 4$, 
then there exist either 
\begin{enumerate}[{\rm (a)}]
\item
$x_1, x_2, x_3, x_4\in X$ such that 
$w_i\sd w_j=x_i\sd x_j$ 
for any $i< j\le 4$ 
or 
\item
$x_1, x_2, x_3 \in X$ such that 
$$
\begin{aligned}
w_1\sd w_4&= w_2\sd w_3= x_2\sd x_3,\\ 
w_2\sd w_4&=w_1\sd w_3 = x_1\sd x_3,\\ 
w_3\sd w_4&=w_1\sd w_3= x_1\sd x_3.
\end{aligned}
$$
\end{enumerate}
\item
If $w_1, w_2, \dots\in B(X)$ and $w_i\sd w_j \in B_2(X)$ for any $i< j$, 
then there exist $x_1, x_2, \dots \in X$ such that $ w_i\sd w_j=x_i\sd x_j$ for any $i< j$. 
\end{enumerate}
\end{lemma}

\begin{proof}
We prove the lemma by induction on $k$. 
There is nothing to prove for $k = 1$, and for $k = 2$, assertion~(i) 
obviously holds. 

Suppose that $k=3$. For some 
$y_1, y_2, y_3, y_4\in X$, we have $w_1 \sd w_2 = y_1\sd y_2$ and  $w_2 \sd w_3 = y_3\sd y_4$. 
Since $w_1\sd w_3 = w_1 \sd w_2  \sd w_2 \sd w_3\in B_{=2}(X)$, it follows that either $y_1=y_3$, 
$y_1=y_4$,  $y_2= y_3$, or $y_2= y_4$. If $y_1 = y_3$, then $w_1\sd w_3= y_2\sd y_4$ 
and $w_2\sd w_3=y_1\sd y_4$, so that we can set $x_1= y_2$, $x_2=y_1$, and $x_3=y_4$. 
If $y_2 = y_3$, then $w_1\sd w_3= y_1\sd y_4$ 
and $w_2\sd w_3=y_2\sd y_4$, and  we set $x_1= y_1$, $x_2=y_1$, and $x_3=y_4$. The remaining cases 
are treated similarly.

Suppose that $k= 4$ and let $x_1, x_2, x_3\in X$ be such that 
$w_i\sd w_j=x_i\sd x_j$ for $i=1, 2, 3$. 
There exist $y, z\in X$ for which $w_1\sd w_4 = y\sd z$. We have $w_2\sd w_4=
w_1\sd w_2\sd w_1\sd w_4 = x_1\sd x_2\sd y\sd z\in B_2(X)$. Therefore, either 
$x_1 = y$, $x_2 = y$, $x_1=z$, or $x_2=z$. 

If $x_1= y$ or $x_1= z$, then the condition in~(ii)\,(a) holds
for $x_4 = z$ in the former case and $x_4 = y$ in the latter. 

Suppose that $x_1\ne y$ and $x_1\ne z$. 
Then $x_2=y$ or $x_2=z$. Let $x_2=y$. Then $w_1\sd w_4=x_2\sd z$, and 
we have  $w_3\sd w_4=
w_1\sd w_3\sd w_1\sd w_4 = x_1\sd x_3\sd x_2\sd z\in B_2(X)$, whence $x_3=z$ (because $x_1,x_2\ne z$), 
so that $w_1\sd w_4=x_2\sd x_3 = w_2\sd w_3$, $w_2\sd w_4 = w_1\sd w_2\sd w_1\sd w_4= 
x_1\sd x_3=w_1\sd w_3$, and $w_3\sd w_4= x_1\sd x_3=w_1\sd w_3$, i.e., assertion~(ii)\,(b) holds. 
The case $x_2=z$ is similar. Note for what follows that, in both cases $x_2=y$ and $x_2=z$, we have 
$w_4=w_1\sd w_2\sd w_3$. 

Let $k> 4$. Consider the words $w_1$, $w_2$, $w_3$, and $w_4$. Let 
$x_1, x_2, x_3\in X$ be such that $w_i\sd w_j=x_i\sd x_j$ for $i=1, 2, 3$. 
As previously, there exist $y, z\in X$ for which $w_1\sd w_4 = y\sd z$ and 
either $x_1 = y$, $x_2 = y$, $x_1=z$, or $x_2=z$.   

Suppose that $x_1\ne y$ and $x_1\ne z$; then $w_4=w_1\sd w_2 \sd w_3$. In this case, 
we consider $w_5$ instead 
of $w_4$. Again, there exist $y', z'\in X$ for which $w_1\sd w_5 = y\sd z$ and 
either $x_1 = y'$, $x_2 = y'$, $x_1=z'$, or $x_2=z'$. Since $w_5\ne w_4$, it follows that $w_5\ne 
w_1\sd w_2\sd w_3$, and we have $x_1 = y'$ or $x_1=z'$. In the former case, we set $x_5 = z'$ and 
in the latter, $x_5 = y'$. Consider again $w_4$; recall that $w_1\sd w_4 = y\sd z$. We have 
$w_i\sd w_4= w_1\sd w_i\sd w_1\sd w_4 = x_1\sd x_i\sd y\sd z\in B_{=2}(X)$ for $i \in \{2, 3, 5\}$. 
Since $x_2\ne x_5$ and $x_3\ne x_5$, it follows that $x_1 = y$, which contradicts the assumption.  

Thus, $x_1= y$ or $x_1 = z$. As above, we set $x_4 = z$ in the former case and $x_4 = y$ in the 
latter; then the condition in~(ii)\,(a) holds. 

Suppose that we have  already found the required 
$x_1, \dots, x_{k-1}\in X$ for $w_1, \dots,\allowbreak w_{k-1}$. There exist $y, z\in X$ 
for which $w_1\sd w_k = y\sd z$. We have $w_i\sd w_k=
w_1\sd w_i\sd w_1\sd w_k = x_1\sd x_i\sd y\sd z\in B_{=2}(X)$ for $i \le k-1$. If $x_1\ne y$ and 
$x_1\ne z$, then we have $x_i \in \{y, z\}$ for $2\le i\le k-1$, which is impossible, because 
$k>4$. Thus, either $x_1=y$ or $x_1 = z$. In the former case, we set $x_k = z$ 
and in the latter, $x_k = y$. Then $w_1\sd w_k= 
x_1\sd w_k$ and, for any $i\le k-1$, $w_i\sd w_k=w_1\sd w_i\sd w_1\sd w_k= x_1\sd x_i\sd x_1\sd x_k 
= x_i\sd x_k$. 

The infinite case is proved by the same inductive argument. 
\end{proof}

\begin{proof}[Proof of Theorem~\ref{B2(X)}]
(i)\enspace Suppose that $\mathscr F$ is a $k$-arrow filter on $X$.
Let $C$ be a $k$-fat set in $B(X_{\mathscr F})$. 
Consider the $2$-coloring of $[X]^2$ defined by 
$$
c(\{x,y\}) =\begin{cases} 0& \text{if $\{x, y\}=x\sd y\in C$}, \\ 1&\text{otherwise}.
\end{cases}
$$
Since $\mathscr F$ is $k$-arrow, there exists either an $A\in \mathscr F$ 
for which $c([A]^2)=\{0\}$ and hence $[A]^2\subset C\cap B_2(X_{\mathscr F})$ 
or a $k$-element set 
$F\subset X$ for which $c([F]^2)=\{1\}$ and hence $[F^2]\cap C = [F^2]\cap C\cap 
B_{=2}(X_{\mathscr F})=\varnothing$. The latter case cannot occur, 
because $C$ is $k$-fat. Therefore, $C\cap B_2(X_{\mathscr F})$ contains the trace 
$[A]^2\cup\{\zero\}= ((A\cup \{*\})\sd (A\cup \{*\}))$ 
of the subgroup $\langle A\cup \{*\}\rangle$, which is an open neighborhood 
of zero in $B(X_{\mathscr F})$. 

Now suppose that $k\ne  4$ and the trace of each $k$-fat set on $B_2(X)$ 
contains the trace on $B_2(X)$ 
of a neighborhood of zero in $B(X_{\mathscr F})$, i.e., a set of the form $A\sd A$ 
for some $A\in \mathscr F$. Let us show that $\mathscr F$ is $k$-arrow. Given any 
$c\colon [X]^2\to \{0, 1\}$, we set 
$$
C= \bigl\{x\sd y: c(\{x, y\})= 1\bigr\}\quad \text{and}\quad 
C'=B(X_{\mathscr F})\setminus C.
$$
If $C'$ is not $k$-fat, then there exist 
$w_1, \dots, w_k\in B(X)$ such that $w_i\sd w_j\in C$ for $i<j\le k$. 
By Lemma~\ref{2-words}\,(i) 
we  can find $x_1, \dots, x_k \in X$  such that $x_i\sd x_j\in C$ (and hence $x_i\ne *$) 
for $i<j\le k$. This means 
that, for $F=\{x_1, \dots, x_k\}$, we have $c([F]^2)=\{1\}$. If $C'$ is $k$-fat, then, 
by assumption, there exists an $A\in \mathscr F$ for which $A\sd A\setminus \{\zero\} \subset 
C'\cap B_2(X) = C$, which means that $c([A]^2)=\{0\}$.

The same argument proves~(ii); the only difference is that assertion~(ii) 
of Lemma~\ref{2-words} is used instead of~(i). 

The proof of~(iii) is similar.
\end{proof}

Let $R_{r}(s)$ denote the least number $n$ such that, for any $r$-coloring  
$c\colon [X]^2\to Y$, where $|X|\ge n$ and $|Y|=r$, 
there exists an $s$-element $c$-homogeneous set. By the finite Ramsey theorem, such a number 
exists for any positive integers $r$ and $s$.

\begin{theorem}
\label{B4(X)}
There exists a positive integer $N$ (namely, $N = R_{36}(R_6(3))+1$) such that, 
for any uniform ultrafilter $\mathscr U$ on a set $X$ of infinite cardinality $\kappa$, 
the following conditions are equivalent:
\begin{enumerate}[{\rm (i)}]
\item 
the trace of any $N$-fat subset 
of $B(X)$ on $B_4(X)\subset B_4(X_{\mathscr U})$ contains that 
of a neighborhood of zero in the free group topology of $B(X_{\mathscr U})$; 
\item 
all $\kappa$-fat sets in $B(X)$ 
are neighborhoods of zero 
in the topology induced from the free topological group $B(X_{\mathscr U})$;
\item 
$\mathscr U$ is a Ramsey ultrafilter.
\end{enumerate}
\end {theorem}

\begin{proof}
Without loss of generality, we assume that $X=\kappa$.

(i) $\Rightarrow$ (iii)\enspace
Suppose that $N$ is as large as we need and  the trace of each 
$N$-fat set on $B_{4}(\kappa_{\mathscr U})$ contains 
the trace on $B_{4}(\kappa)$ 
of a neighborhood of zero in $B(X_{\mathscr U})$, which, in turn, 
contains a set of the form 
$(A\sd A\sd A\sd A)\cap B_{=4}(\kappa)$
for some $A\in \mathscr U$. Let us show that $\mathscr U$ is a Ramsey ultrafilter. 
Consider any 2-coloring $c\colon [\kappa]^2\to \{0, 1\}$.
We set 
\begin{multline*}
C= \bigl\{\alpha_1\sd \alpha_2\sd \alpha_3\sd \alpha_4: \alpha_i\in \kappa 
\text{ for $i\le 4$},\ \alpha_1<\alpha_2<\alpha_3<\alpha_4,\\  
c(\{\alpha_1, \alpha_2\})\ne c(\{\alpha_3, \alpha_4\}),\ 
c(\{\alpha_1, \alpha_3\})\ne c(\{\alpha_2, \alpha_4\}),\\
c(\{\alpha_1, \alpha_4\})\ne c(\{\alpha_2, \alpha_3\})\bigr\}
\end{multline*}
and 
$$
C'=B(X)\setminus C.
$$
If $C'$ is not $N$-fat, then there exist 
$w_1, \dots, w_N\in B(\kappa)$ such that $w_i\sd w_j\in C$ for $i<j\le N$. We can 
assume that $w_N=\zero$ (otherwise, we translate all $w_i$ by $w_N$). Then 
$w_i\in C\subset B_4(\kappa)$, $i<N$. Let $w_i=\alpha^i_1\sd \alpha^i_2\sd \alpha^i_3 
\sd \alpha^i_4$ for $i<N$ 
and consider the 36-coloring of all pairs $\{w_i, w_j\}$, $i<j<N$,  defined as follows. 
Since $w_i\sd w_j$ is a four-letter word, it follows that $w_i\sd w_j = \beta_1\sd 
\beta_2\sd \beta_3\sd \beta_4$, 
where $\beta_i\in \kappa$. Two letters  among $\beta_1, \beta_2, \beta_3, 
\beta_4$ (say $\beta_1$ and $\beta_2$) 
occur in the word $w_i$ and the remaining two ($\beta_3$ and $\beta_4$) 
occur in $w_j$. We assume 
that $\beta_1<\beta_2$ and $\beta_3< \beta_4$. Let us denote 
the numbers of the letters  $\beta_1$ and $\beta_2$ in $w_i$ 
(recall that the letters in $w_i$ are numbered 
in increasing order) by $i'$ and $i''$, respectively, 
and the numbers of the letters $\beta_3$ and $\beta_4$ in $w_j$ by $j'$ and $j''$. To the pair 
$\{w_i, w_j\}$ we assign the quadruple $(i', i'', j', j'')$. The number of all possible quadruples 
is 36, so that this assignment is a 36-coloring. We choose $N\ge R_{36}(N')+1$ for 
$N'$ as large as we  need. Then there exist two 
pairs $i'_0, i''_0$ and $j'_0, j''_0$ and  $N'$ 
words $w_{i_n}$, where $n\le N'$ and $i_s<i_t$ for $s<t$, 
such that $i'=i'_0$, $i''=i''_0$,  $j'=j'_0$, 
and $j''=j''_0$ for any 
pair $\{w_i, w_j\}$ with $i, j\in \{i_1, \dots, i_{N'}\}$ and $i<j$. Clearly, 
if $N'\ge 3$, then we also have $j'_0=i'_0$ and $j''_0= i''_0$. In the same manner, we 
can fix the position of the letters coming from $w_i$ and $w_j$ in the sum $w_i\sd w_j$: to each 
pair $\{w_{i_s}, w_{i_t}\}$, $s,t\in \{1, \dots, N'\}$, $s<t$, we assign the numbers of the 
$i'_0$th and $i''_0$th letters of $w_{i_s}$ in the word $w_{i_s}\sd w_{i_t}$ (recall that 
the letters of all words are numbered in increasing order); 
the positions of the letters of $w_{i_t}$ in $w_{i_s}\sd w_{i_t}$ are then determined 
automatically. There are six possible arrangements: $1,2$, $1,3$, $1,4$, $2,3$, 
$2,4$, and $3,4$. Thus, we have a 6-coloring of the symmetric square of the 
$N'$-element set $\{w_{i_1}, \dots, w_{i_{N'}}\}$, and if $N'\ge R_6(3)$ (which we assume), 
then there exists 
a 3-element set $\{w_k, w_l, w_m\}$ homogeneous with respect to this coloring, i.e., such that 
all pairs of words from this set are assigned the same color. 
For definiteness, suppose that this is the color $1,2$; suppose also that $i'_0=1$, $i''_0=2$, 
$k<l<m$, 
and $w_t= \alpha^t_1\sd \alpha^t_2\sd \alpha^t_3\sd \alpha^t_4$ for $t= k, l, m$. Then 
$w_k, w_l, w_m\in C$, $w_k\sd  w_l= \alpha^k_1\sd \alpha^k_2\sd \alpha^l_1\sd \alpha^l_2\in C$, 
$w_l\sd  w_m= \alpha^l_1\sd \alpha^l_2\sd \alpha^m_1\sd \alpha^m_2\in C$, and 
$w_k\sd  w_m= \alpha^k_1\sd \alpha^k_2\sd \alpha^m_1\sd \alpha^m_2\in C$. 
By the definition of $C$ we have 
$c(\alpha^k_1\sd \alpha^k_2)\ne c(\alpha^l_1\sd \alpha^l_2)$, 
$c(\alpha^l_1\sd \alpha^l_2)\ne c(\alpha^m_1\sd \alpha^m_2)$, and 
$c(\alpha^k_1\sd \alpha^k_2)\ne c(\alpha^m_1\sd \alpha^m_2)$, which is impossible, because 
$c$ takes only two values. The cases of other colors and other numbers 
$i'_0$ and $i''_0$ are treated in a similar way. 

Thus, $C'$ is $N$-fat and, therefore, contains $(A\sd A\sd A\sd A) \cap B_4(\kappa)$ for some 
$A\in \mathscr U$. Take any $\alpha \in A$ and consider the sets 
$A'=\{\beta>\alpha: c(\{\alpha, \beta\}) = \{0\}\}$ and 
$A''=\{\beta>\alpha: c(\{\alpha, \beta\}) = \{1\}\}$. One of these sets belongs to $\mathscr U$, 
because $\mathscr U$ is uniform. For definiteness, suppose that this is $A'$. By 
Theorem~\ref{B2(X)} $\mathscr U$ is 3-arrow. Therefore, there exists either an $A''\subset A'$ 
for which $c([A'']^2)=\{0\}$ or $\beta, \gamma, \delta\in A'$, $\beta<\gamma<\delta$, for which 
$c([\{\beta, \gamma, \delta\}]^2)=\{1\}$. In the former case, we are done. In the latter case, we 
have $\alpha,\beta,\gamma,\delta\in A$, $\alpha <\beta<\gamma<\delta$, 
$c(\{\beta, \gamma\})= c(\{\gamma,\delta\})=c(\{\beta,\delta\})=1$, and 
$c(\{\alpha, \beta\})=c(\{\alpha, \gamma\})=c(\{\alpha, \delta\})=0$ (by the definition of 
$A'$).  Therefore, $\alpha \sd \beta\sd \gamma\sd \delta\in C$, which contradicts the definition 
of $A$.

(iii) $\Rightarrow$ (ii)\enspace 
Suppose that $\mathscr U$ is a Ramsey ultrafilter on $X$ and $C$ is a $\kappa$-fat set 
in $B(X)$. Take any $n\in \mathbb N$ and consider the coloring 
$c\colon [X]^{2n}\to \{0, 1\}$ defined by
$$
c(\{x_1, \dots, x_{2n}\}) = \begin{cases}
0&\text{if $\{x_1, \dots, x_{2n}\}=x_1\sd  \dots \sd x_{2n}\in C$}, \\ 
1&\text{otherwise}.
\end{cases}
$$
Since $\mathscr U$ is Ramsey, there exists either a set $A_{n}\in \mathscr U$ 
for which $[A]^{2n}\subset C$ 
or a set $Y\subset X$ of cardinality $\kappa$ for which $[Y]^{2n}\cap C =\varnothing$. 
In the latter case, for  $Z=[Y]^n\subset B(X)$, we have 
$(Z\sd Z) \cap C\subset \{\zero\}$, which contradicts $C$ being $\kappa$-fat.    
Hence the former case occurs, and $C\cap B_{2n}(X)$ contains the trace 
$[A_{n}]^{2n} \cap B_{=2n}(X)$ 
of the open subgroup $\langle (A_{n}\cup\{*\})\sd (A_{n}\cup\{*\})\rangle$ 
of $B(X_{\mathscr F})$.

Thus, for each $n\in \mathbb N$, we have found $A_1, A_2, \dots, A_n\in \mathscr F$ such that 
$[A_{i}]^{2i} \cap B_{=2i}(X) \subset C$. Let $A=\bigcap_{i\le n} A_i$. Then $A \in \mathscr U$ 
and $[A]^{2i} \cap B_{=2i}(X) \subset C$ for all $i\le n$. Hence $C\cap B_{2n}(X)$ 
contains the trace on $B_{2n}(X)$ of the open subgroup 
$\langle (A\cup\{*\})\sd (A\cup\{*\})\rangle$ of $B(X_{\mathscr U})$ (recall that 
$\zero\in C$). This means that, for each $n$,  
$C\cap B_{2n}(X)$ is a neighborhood of zero in the topology induced from 
$B(X_{\mathscr U})$.

If $\kappa =\omega$, then $B(X_{\mathscr U})$ has the inductive limit topology with respect 
to the decomposition $B(X_{\mathscr U})=\bigcup_{n\in \omega}B_n(X_{\mathscr F})$, because 
$\mathscr F$ is Ramsey (see \cite{axioms}). 
Therefore, in this case, $C\cap B(X)$ is a neighborhood of zero in the induced topology. 

If $\kappa>\omega$, then the ultrafilter $\mathscr U$ is countably complete 
\cite[Lemma~9.5 and Theorem~9.6]{Comfort-Negrepontis}, i.e., any countable intersection 
of elements of $\mathscr U$ belongs to $\mathscr U$. Hence 
$A=\bigcap_{n\in \mathbb N} A_n\in \mathscr U$, 
and $\langle (A\cup\{*\})\sd (A\cup\{*\})\rangle \cap 
\Bigl(\bigcup_{n\in \omega}B_{2n}(X)\Bigr)\subset C$. Thus, 
$C\cap B(X)$ is a neighborhood of zero in the induced topology in this case, too. 

 The implication (ii)~$\Rightarrow$~(i) is obvious. 
\end{proof}

Theorem~\ref{B2(X)} has the following purely algebraic corollary. 

\begin{corollary} 
[$\mathfrak p=\mathfrak c$]
Any Boolean group contains $\omega$-fat sets 
which are  not fat and $\Delta^*$-sets which are $\Delta^*_k$-sets for no $k$. 
\end{corollary}

\begin{proof}
Theorem~4.10 of \cite{Baumgartner-Taylor} asserts that 
if $\mathfrak p=\mathfrak c$, then there exists an ultrafilter $\mathscr U$ 
on $\omega$ which is $k$-arrow for all 
$k\in \mathbb N$ but not Ramsey and, therefore, not $\omega$-arrow 
\cite[Theorem~2.1]{Baumgartner-Taylor}. By Theorem~\ref{B2(X)} the traces of all fat sets on 
$B_2(\omega)$ contain those of neighborhoods of zero in $B(\omega_{\mathscr U})$, and there exist 
$\omega$-fat sets whose traces do not. This proves the required assertion for 
the countable Boolean group. The case of a group $B(X)$ of uncountable cardinality $\kappa$ reduces 
to the countable case by representing $B(X)$ as $B(\kappa)=B(\omega)\times B(\kappa)$; it suffices 
to note that a set of the form $C\times B(\kappa)$, where $C\subset B(\omega)$, is $\lambda$-fat in 
$B(\omega)\times B(\kappa)$ for $\lambda\le \omega$ if and only if so is $C$ in $B(\omega)$. 
\end{proof}

The author is unaware of where there exist ZFC examples of such sets in any groups. 

\section*{Acknowledgments}

The author is very grateful to Evgenii Reznichenko and Anton Klyachko for useful discussions.

\end{document}